\documentclass[11pt,reqno]{amsart}

\usepackage{amsrefs}
\usepackage[all]{xy}

\newtheorem{Thm}{Theorem}[section]
\newtheorem{Cor}[Thm]{Corollary}
\newtheorem{Lem}[Thm]{Lemma}
\newtheorem{Prop}[Thm]{Proposition}

\newtheorem{step}{Step}[section]

\theoremstyle{definition}
\newtheorem{Rem}[Thm]{Remark}

\newcommand{\bbA}{\mathbb{A}}
\newcommand{\bbM}{\mathbb{M}}
\newcommand{\bbN}{\mathbb{N}}

\newcommand{\cI}{\mathcal{I}}
\newcommand{\cJ}{\mathcal{J}}
\newcommand{\cO}{\mathcal{O}}
\newcommand{\cP}{\mathcal{P}}
\newcommand{\cS}{\mathcal{S}}
\newcommand{\cU}{\mathcal{U}}
\newcommand{\cV}{\mathcal{V}}

\newcommand{\bd}{\mathbf{d}}
\newcommand{\bp}{\mathbf{p}}
\newcommand{\bq}{\mathbf{q}}

\newcommand{\vep}{\varepsilon}

\newcommand{\df}{\colon}

\let\mod=\undefined
\DeclareMathOperator{\GL}{GL}

\DeclareMathOperator{\End}{End}

\DeclareMathOperator{\Hom}{Hom}
\DeclareMathOperator{\mod}{mod}
\DeclareMathOperator{\rad}{rad}
\DeclareMathOperator{\rep}{rep}
\DeclareMathOperator{\supp}{supp}
\DeclareMathOperator{\gldim}{gl.dim}
\DeclareMathOperator{\bdim}{\mathbf{dim}}

\newcommand{\ol}{\overline}

\begin{document}


\title[Algebras with irreducible module varieties II]{Algebras with irreducible module varieties II: Two vertex case}

\author{Grzegorz Bobi\'nski}
\address{Grzegorz Bobi\'nski\newline
Faculty of Mathematics and Computer Science\newline
Nicolaus Copernicus University\newline
ul. Chopina 12/18\newline
87-100 Toru\'n\newline
Poland}
\email{gregbob@mat.umk.pl}

\author{Jan Schr\"oer}
\address{Jan Schr\"oer\newline
Mathematisches Institut\newline
Universit\"at Bonn\newline
Endenicher Allee 60\newline
53115 Bonn\newline
Germany}
\email{schroer@math.uni-bonn.de}

\begin{abstract}
We call a finite-dimensional $K$-algebra $A$ \emph{geometrically irreducible} if for all $d \ge 0$ all connected components of the affine scheme of $d$-dimensional $A$-modules are irreducible. We prove that a geometrically irreducible algebra with exactly two simple modules has to be of a very special form, which we describe. Based on this result we prove that every minimal geometrically irreducible algebra without shortcuts in its Gabriel quiver has at most two simple modules.
\end{abstract}



\maketitle

\setcounter{tocdepth}{1}
\numberwithin{equation}{section}
\tableofcontents


\section{Introduction and main result}


\subsection{Introduction}
Throughout the paper $K$ is an algebraically closed field. By an algebra we mean a finite-dimensional algebra over $K$ and by a module a finite-dimensional left module.

Given an algebra $A$ and a nonnegative integer $d$, one defines the variety $\mod (A, d)$, called the variety of $d$-dimensional $A$-modules, which consists of the $K$-algebra homomorphisms $A \to \bbM_d(K)$ and parameterizes $A$-modules of dimension $d$. Different aspects of its geometry have been objects of intensive studies during the last decades (see~\cites{BobinskiRiedtmannSkowronski, Bongartz1998, Geiss, Zwara} for some references and reviews of results). In particular, one may ask, for a given algebra $A$, which properties the varieties $\mod (A, d)$ have for all dimensions $d$. For example, Bongartz has proved in~\cite{Bongartz1991} that the varieties $\mod (A, d)$ are smooth for all $d$ if and only if $A$ is hereditary (i.e., $\gldim A \leq 1$).

In~\cite{BobinskiSchroer2017a} we initiated the study of geometrically irreducible algebras, i.e., the algebras such that for each $d$ the connected components of $\mod (A, d)$ are irreducible. In particular, we proved that the only geometrically irreducible algebras of finite global dimension are the hereditary ones. On the other hand, the truncated polynomial algebras are the only geometrically irreducible local algebras. Note that an algebra $A$ is Morita equivalent to a local algebra if and only if there is exactly one (up to isomorphism) simple $A$-module. In this paper we concentrate on algebras with exactly two simple modules.

\subsection{Main result}
In order to formulate the main result of the paper we need to introduce some families of algebras. For $h \geq 0$, let $Q (h)$ be the quiver
\[
\xymatrix{
0 \ar@(ul,dl)_{\vep_0} & 1 \ar@(ur,dr)^{\vep_1} \ar@<-1ex>[l]_{\alpha_1} \ar@<1ex>^{\alpha_h}_{\cdots}[l]
}.
\]
Note that $Q (h)$ is connected if and only if $h > 0$. For $n \geq 1$ (and $h \geq 1$), let
\[
\rho^{(n)} := \sum_{i = 0}^n \vep_0^{n - i} \alpha_1 \vep_1^i.
\]
For $m_0, m_1 \geq 2$ and $h, n \geq 1$, let $A (h, m_0, m_1, n)$ be the path algebra of the quiver $Q (h)$ bounded by
\[
\vep_0^{m_0}, \; \vep_1^{m_1}, \; \rho^{(n)}.
\]
Similarly, for $m_0, m_1 \geq 1$ and $h \geq 0$, let $A' (h, m_0, m_1)$ be the path algebra of the quiver $Q (h)$ bounded by
\[
\vep_0^{m_0}, \; \vep_1^{m_1}.
\]

The following theorem is the first main result of the paper.

\begin{Thm}\label{thm:intro3}
Assume that $A$ is an algebra which has exactly two simples. If $A$ is geometrically irreducible, then $A$ is Morita equivalent to one of the following algebras:
\begin{enumerate}

\item
$A (h, m, m, 1)$, for some $m \geq 2$ and $h \geq 1$, or

\item
$A (h, m, m, m - 1)$, for some $m \geq 2$ and $h \geq 1$, or

\item
$A' (h, m_0, m_1)$, for some $m_0, m_1 \geq 1$ and $h \geq 0$.

\end{enumerate}
\end{Thm}

It is easy to see that if $m_0, m_1 \geq 1$ and $h \geq 0$, then $A' (h, m_0, m_1)$ is geometrically irreducible (see Proposition~\ref{prop Aprim}).

Now let $A = A (h, m, m, 1)$, for $m \geq 2$ and $h \geq 1$. The category $\mod(A)$ is related to the Birkhoff problem~\cite{Birkhoff} and has been intensively studied (see for example~\cites{KussinLenzingMeltzer, RingelSchmidmeier, Simson}). In~\cite{BobinskiSchroer2017b} we prove that $A (h, m, m, 1)$ is geometrically irreducible.

We do not know if $A (h, m, m, m - 1)$ is geometrically irreducible, if $m \geq 3$ and $h \geq 1$.

\subsection{Minimal geometrically irreducible algebras without shortcuts}
Our interest in two-vertex case and its significance follows from a conjecture formulated in~\cite{BobinskiSchroer2017a} concerning a general form of geometrically irreducible algebras. We call an algebra $A$ a \emph{gluing} of bound quiver algebras $K Q' / \cI'$ and $K Q'' / \cI''$ if $A$ is Morita equivalent to the algebra $K Q / \cI$, where
\begin{enumerate}

\item
$Q_0 = Q'_0 \cup Q''_0$ (the union is not necessarily disjoint);

\item
$Q_1 = Q'_1 \cup Q''_1$ and $Q'_1 \cap Q''_1 = \emptyset$;

\item
$\cI$ is the ideal in $K Q$ generated by the union $\cI' \cup \cI''$.

\end{enumerate}
A connected algebra, which cannot be presented as the gluing of two algebras as above with $Q_1' \neq \emptyset \neq Q_1''$, is called \emph{minimal}. We conjecture that every geometrically irreducible minimal algebra has at most two simple modules. The second main result of the paper says this conjecture holds for a wide class of algebras. More precisely, an arrow $\alpha$ in a quiver $Q$ is called a \emph{shortcut}, if $\alpha$ is not a loop and there is a path $\alpha_1 \cdots \alpha_l$ of length at least~2 connecting the same vertices as $\alpha$, such that none of the arrows $\alpha_i$ is a loop. Then we have the following.

\begin{Thm}\label{thm:conj}
Assume that $A = K Q / \cI$ is a minimal geometrically irreducible algebra. If there are no shortcuts in $Q$, then $Q$ has at most two vertices, hence $A$ is Morita equivalent to one of the following algebras:
\begin{enumerate}

\item
$K [X] / (X^m)$, for some $m \geq 1$, or

\item
$A (1, m, m, 1)$, for some $m \geq 2$, or

\item
$A (1, m, m, m - 1)$, for some $m \geq 2$.

\end{enumerate}
\end{Thm}

Recall we know that $K [X] / (X^m)$, for $m \geq 1$, and $A (1, m, m, 1)$, for $m \geq 2$, are geometrically irreducible, while it is open if $A (1, m, m, m - 1)$, for $m \geq 3$, is geometrically irreducible.

\subsection{Acknowledgments}
The first author acknowledges the support of the National Science Center grant no.\ 2015/17/B/ST1/01731. He would also like to thank the University of Bonn for its hospitality during his two one week visits in July 2016 and July 2017. The second author thanks the Nicolaus Copernicus University in Toru\'n for one week of hospitality in October 2015, where this work was initiated. He thanks for a second week in Toru\'n in March 2017, and he is also grateful to the SFB/Transregio TR 45 for financial support.


\section{Schemes of representations}


\subsection{Quivers with relations} \label{subsection quivers}
By a \emph{quiver} we mean a quadruple $Q = (Q_0, Q_1, s, t)$, where $Q_0$ and $Q_1$ are finite sets of vertices and arrows, respectively, and $s, t \df Q_1 \to Q_0$ are maps. We say that an arrow $\alpha$ starts in $s (\alpha)$ and ends in $t (\alpha)$.

A \emph{path} of length $l \geq 1$ in a quiver $Q$ is a sequence $\sigma = (\alpha_1, \ldots, \alpha_l)$ of arrows $\alpha_1, \ldots, \alpha_l \in Q_1$ such that $s (\alpha_i) = t (\alpha_{i + 1})$ for all $1 \leq i \leq l - 1$. We usually write $\alpha_1 \cdots \alpha_l$ instead of $(\alpha_1, \ldots, \alpha_l)$. Define $s (\sigma) := s (\alpha_l)$ and $t (\sigma) := t (\alpha_1)$. A path $\sigma$ of positive length is called an \emph{oriented cycle} if $s (\sigma) = t (\sigma)$. A \emph{loop} is an oriented cycle of length one, i.e., an arrow $\alpha$ such that $s (\alpha) = t (\alpha)$. Additionally to the paths of length $l \geq 1$, there is, for each vertex $i \in Q_0$, a path $e_i$ of length $0$ with $s (e_i) = t (e_i) = i$.

By $K Q$ we denote the path algebra of a quiver $Q$. It has all paths (including the length $0$ paths) as a $K$-basis, and the multiplication is defined via the composition of paths. Note that $K Q$ may be infinite-dimensional.

By a \emph{relation} in a quiver $Q$ we mean a linear combination of paths of length at least~2, which have the same starting vertex and the same ending vertex. An ideal $\cI$ in $K Q$ is called \emph{admissible} if the following hold:
\begin{itemize}

\item[(i)]
$\cI$ is generated by a set of relations;

\item[(ii)]
there exists some $l \geq 2$ such that all paths of length at least $l$ are contained in $\cI$.

\end{itemize}
A pair $(Q, \cI)$ consisting of a quiver $Q$ and an admissible ideal $\cI$ is called a \emph{bound quiver}. If $(Q, \cI)$ is a bound quiver, then we call $A := K Q / \cI$ a bound quiver algebra. One easily observes that $|Q_0|$ is the number of simple $A$-modules. If $i \in Q_0$, the we put $A_i := e_i A e_i$.

For each algebra $A$ there exists a bound quiver algebra $B$ such that $A$ and $B$ are Morita equivalent. It follows from Bongartz~\cite{Bongartz1991} that if algebras $A$ and $B$ are Morita equivalent, then $A$ is geometrically irreducible if and only if $B$ is geometrically irreducible. Consequently, from now on we assume that all considered algebras are bound quiver algebras.

\subsection{Schemes of representations}
Let $A = K Q / \cI$ be a bound quiver algebra. A \emph{representation} of $Q$ is a tuple $M = (M (i), M (\alpha))$, where  for each vertex $i \in Q_0$, $M (i)$ is a finite-dimensional $K$-vector space and, for each arrow $\alpha \in Q_1$, $M(\alpha) \df M (s (\alpha)) \to M (t (\alpha))$ is a $K$-linear map. The sequence $\bdim M = (\dim_K M (i))_{i \in Q_0}$ is called the \emph{dimension vector} of $M$. A \emph{morphism} $f \colon M \to N$ of representations is a sequence $f = (f (i))$, where for each vertex $i \in Q_0$, $f (i) \df M (i) \to N (i)$ is a $K$-linear map. For a path $\sigma = \alpha_1 \cdots \alpha_l$ in $Q$ of positive length and a representation $M$ we define
\[
M (\sigma) := M(\alpha_1) \cdots M(\alpha_l).
\]
Similarly, if $\rho = \sum_{i = 1}^k \lambda_i \sigma_i$ is linear combination of paths of positive lengths having the same starting vertex
and the same ending vertex,
we set
\[
M (\rho) := \sum_{i = 1}^k \lambda_i M (\sigma_i).
\]
A representation $M$ of $Q$ is a \emph{representation of $A$} (shortly an $A$-representa\-tion) if $M (\rho) = 0$ for all relations $\rho \in \cI$.

For a dimension vector $\bd = (d_i)_{i \in Q_0} \in \bbN^{Q_0}$ let
\[
\rep (Q, \bd) :=  \prod_{\alpha \in Q_1} \Hom_K (K^{d_{s (\alpha)}}, K^{d_{t (\alpha)}})
\]
be the affine space of representations $M$ of $Q$ with $M(i) = K^{d_i}$ for each $i \in Q_0$ and let
\begin{multline*}
\rep (A, \bd) := \{ M = (M(\alpha)) \in \rep(Q,\bd) \mid
\\
\text{$M (\rho) = 0$ for all relations $\rho \in \cI$} \}.
\end{multline*}
Then $\rep(A, \bd)$ is an affine scheme, called the \emph{scheme of $A${}-represen\-tations} of dimension vector $\bd$. Using Bongartz~\cite{Bongartz1991} again, we get that $A$ is geometrically irreducible if and only if $\rep (A, \bd)$ is irreducible for each dimension vector~$\bd$.

The group
\[
\GL (\bd) := \prod_{i \in Q_0} \GL_{d_i} (K)
\]
acts on $\rep(A, \bd)$ by conjugation. The $\GL (\bd)$-orbits are in bijection with the isomorphism classes of $A$-representations with dimension vector $\bd$. For a representation $M$ we denote the corresponding orbit by $\cO_M$. Then
\begin{equation} \label{eq orbit}
\dim \cO_M = \dim \GL (\bd) - \dim_K \End_A (M).
\end{equation}
An orbit $\cO_M$ is called \emph{maximal} if $\cO_M$ is not contained in the closure of another orbit.


\section{Preliminary observations}


\subsection{The possible quivers}
We recall some facts on geometrically irreducible algebras from~\cite{BobinskiSchroer2017a}. If $A = K Q / \cI$ is a bound quiver algebra and $M$ is an $A$-representation, then $M$ is called \emph{locally free} if, for each vertex $i \in Q_0$, $M (i)$ is a free $A_i$-module (recall that $A_i := e_i A e_i$).

\begin{Prop} \label{prop known}
Let $A = K Q / \cI$ be a geometrically irreducible bound quiver algebra. Then the following hold:
\begin{enumerate}

\item
Every oriented cycle in $Q$ is a power of a loop.

\item
For each vertex $i \in Q_0$, there is at most one loop $\alpha$ with $s (\alpha) = i = t (\alpha)$.

\item
Projective and injective $A$-representations are locally free.

\end{enumerate}
\end{Prop}

\begin{proof}
(1) is~\cite{BobinskiSchroer2017a}*{Lemma~3.4}, (2) is~\cite{BobinskiSchroer2017a}*{Lemma~3.2}, and (3) follows from~\cite{BobinskiSchroer2017a}*{Lemma~4.1}.
\end{proof}

For $h \geq 0$, let $Q' (h)$, $Q'' (h)$ and $Q''' (h)$ be the following quivers respectively:
\begin{gather*}
\\
\xymatrix{
0 \ar@(ul,dl)_{\vep_0} & 1 \ar@<-1ex>[l]_{\alpha_1} \ar@<1ex>^{\alpha_h}_{\cdots}[l]
}, \qquad
\xymatrix{
0 & 1 \ar@(ur,dr)^{\vep_1} \ar@<-1ex>[l]_{\alpha_1} \ar@<1ex>^{\alpha_h}_{\cdots}[l]
}, \qquad
\xymatrix{
0 & 1 \ar@<-1ex>[l]_{\alpha_1} \ar@<1ex>^{\alpha_h}_{\cdots}[l]
}.
\end{gather*}
We have the following consequence of Proposition~\ref{prop known}.

\begin{Cor} \label{coro quiver}
If $A = K Q / \cI$ is a geometrically irreducible bound quiver algebra with exactly two simples, then $Q$ is one of the quivers $Q (h)$, $Q' (h)$, $Q'' (h)$ and $Q''' (h)$, $h \geq 0$. In particular, $A \cong A' (h, m_0, m_1) / \cJ$, for some $m_0, m_1 \geq 1$ and an ideal $\cJ$.
\end{Cor}

\begin{proof}
The former part follows immediately from Proposition~\ref{prop known}~(1) and~(2). For the latter observe, that if $Q = Q' (h)$, then we take $m_1 = 1$. If $Q = Q'' (h)$, then $m_0 = 1$, while for $Q = Q''' (h)$, $m_0 = 1 = m_1$ (and $\cJ = 0)$. Finally, if $\vep_0$ (resp. $\vep_1$) is present, then there exists $m_0 \geq 2$ (resp. $m_1 \geq 2$) such that $\vep_0^{m_0} \in \cI$ (resp. $\vep_1^{m_1} \in \cI$), since $\cI$ is admissible.
\end{proof}

The following observation shows that $A$ is geometrically irreducible, if $\cJ = 0$.

\begin{Prop} \label{prop Aprim}
If $A \cong A' (h, m_0, m_1)$, for $m_0, m_1 \geq 1$ and $h \geq 0$, then $A$ is geometrically irreducible.
\end{Prop}

\begin{proof}
Observe that in the situation of the proposition, for each dimension vector $\bd = (d_0, d_1)$, we have
\[
\rep (A, \bd) \cong \rep (A_0, d_0) \times \rep (B, \bd) \times \rep (A_1, d_1),
\]
where $B$ is the path algebra of the quiver $Q''' (h)$. Since $B$ is hereditary, while $A_0 = K [X] / (X^{m_0})$ and $A_1 = K [X] / (X^{m_1})$, we know that $B$, $A_0$ and $A_1$ are geometrically irreducible (see for example~\cite{Gerstenhaber} for the claim in case of $A_0$ and $A_1$), hence the claim follows.
\end{proof}

In Section~\ref{sect proof one} we study the case, when $\cJ \neq 0$.

\subsection{Representation theory of the truncated polynomial algebras}
We need to recall well-known facts from the representation theory of the truncated polynomial algebras.

Let $\Lambda = K [X] / (X^m)$, $m \geq 1$. It is well-known that, for given $d \geq 1$, the isomorphism classes of $d$-dimensional $\Lambda$-modules are parameterized by the partitions of $d$ with parts at most $m$ (we denote the set of such partitions by $\cP_m (d)$): a partition $\bp = (p_1, \ldots, p_l) \in \cP_m (d)$ corresponds to a $\Lambda$-module $M_\bp$ with the action of $X$ given by the matrix
\[
J_\bp := \begin{bmatrix}
J_{p_1} & 0 & \cdots & 0
\\
0 & J_{p_2} & \ddots & \vdots
\\
\vdots & \ddots & \ddots & 0
\\
0 & \cdots & 0 & J_{p_l}
\end{bmatrix},
\]
where $J_p$ denotes the nilpotent Jordan matrix of size $p$ (with ones over the diagonal). Obviously, in the situation above we have
\[
M_\bp = M_{(p_1)} \oplus \cdots \oplus M_{(p_l)}.
\]
If $1 \leq p, q \leq m$, then
\begin{equation} \label{eq Hom}
\dim_K \Hom_\Lambda (M_{(p)}, M_{(q)}) = \min \{ p, q \}.
\end{equation}
A partition $\bp \in \cP_m (d)$ is called \emph{maximal} if $\bp = (m, \ldots, m, r)$, for some $1 \leq r \leq m$. A partition $\bp$ is maximal if and only if the orbit $\cO_\bp := \cO_{M_\bp}$ is maximal in $\rep (\Lambda, d)$. Note that~\eqref{eq Hom} together with~\eqref{eq orbit} gives a method for calculating $\dim \cO_\bp$ for $\bp \in \cP_m (d)$. As a consequence we get the following well-known formula, which we need later.

\begin{Lem} \label{Lem orbits}
If $p \geq 2$, then
\[
\dim \cO_{(p)} - \dim \cO_{(p - 1, 1)} = 2.
\]
\end{Lem}

\begin{proof}
We have $\dim_K \End_\Lambda (M_{(p)}) = p$ and
\[
\dim_K \End_\Lambda (M_{(p - 1, 1)}) = (p - 1) + 1 + 1 + 1 = p + 2,
\]
hence the formula follows.
\end{proof}

\subsection{Stratification} \label{sub strat}
The main tool which we use in the proof of Theorem~\ref{thm:intro3} is the stratification by partitions, which we introduce now.

Let $A = A' (h, m_0, m_1) / \cJ$, for $m_0, m_1 \geq 1$, $h \geq 0$, and an ideal $\cJ$. Since $A' (h, m_0, m_1)$ is a monomial bound quiver algebra (i.e., the defining ideal in $Q (h)$ can be generated paths), we may identify $A' (h, m_0, m_1)$ with the subspace of $K Q (h)$ spanned by paths which do not contain subpaths $\vep_0^{m_0}$ and $\vep_1^{m_1}$ (we stress that $A' (h, m_0, m_1)$ is not a subalgebra of $K Q (h)$). Consequently, if we assume that $m_0$, $m_1$ and $h$ are minimal, then $\cJ$ consists of relations $\rho$ such that $s (\rho) = 1$ and $t(\rho) = 0$.

Observe that $A_0$ and $A_1$ are truncated polynomial algebras and if $M$ is an $A$-representation, then $M (0)$ and $M (1)$ together with the maps $M (\vep_0)$ and $M (\vep_1)$, respectively, are $A_0$- and $A_1$-modules, respectively. By abuse of notation we denote these modules by $M (0)$ and $M (1)$. Consequently, if $\bd = (d_0, d_1)$ is a dimension vector, $\bp \in \cP_{m_0} (d_0)$ and $\bq \in \cP_{m_1} (d_1)$, then we may define
\[
\cS_{\bp, \bq} := \{ M \in \rep (A, \bd) \mid \text{$M (0) \cong M_\bp$ and $M (1) \cong M_\bq$} \}.
\]
If $\rho$ is a relation in $Q (h)$, $M' \in \Hom_K(K^{d_0}, K^{d_0})$, $M'' \in \Hom_K(K^{d_1}, K^{d_1})$, and $M_1, \ldots, M_h \in \Hom_K(K^{d_1}, K^{d_0})$, then we put
\[
\rho (M', M'', M_1, \ldots, M_h) := M (\rho),
\]
where $M \in \rep (K Q (h), \bd)$ is defined by
\begin{align*}
M(\alpha) :=
\begin{cases}
M' & \text{if $\alpha = \vep_0$},
\\
M'' & \text{if $\alpha = \vep_1$},
\\
M_i & \text{if $\alpha = \alpha_i$, $i = 1, \ldots, h$}.
\end{cases}
\end{align*}

The following observations will be important.

\begin{Prop} \label{prop strat}
Let $A$, $\bd$, $\bp$ and $\bq$ be as above.
\begin{enumerate}

\item \label{item strat1}
$\cS_{\bp, \bq}$ is an irreducible locally closed subset of $\rep (A, \bd)$ of dimension
\[
\dim \cO_\bp + \dim \cO_\bq + \dim \bbA_{\bp, \bq},
\]
where
\begin{multline*}
\bbA_{\bp, \bq} := \{ (M_1, \ldots, M_h) \in (\Hom_K (K^{d_1}, K^{d_0}))^h \mid
\\
\text{$\rho (J_\bp, J_\bq, M_1, \ldots, M_h) = 0$ for each $\rho \in \cJ$} \}.
\end{multline*}

\item \label{item strat2}
If $\bp$ and $\bq$ are maximal, then $\ol{\cS}_{\bp, \bq}$ is an irreducible component of $\rep (A, \bd)$.

\end{enumerate}
\end{Prop}

\begin{proof}
See~\cite{BobinskiSchroer2017a}*{Subsection~5.3}.
\end{proof}

Let $\bd = (d_0,d_1)$ be a dimension vector.
If $\bp_0 \in \cP_{m_0} (d_0)$ and $\bq_0 \in \cP_{m_1} (d_1)$ are maximal, then we say that $(\bp_0, \bq_0)$ is a \emph{maximal pair of partitions} for $\bd$ and put $\cS_{\max} := \cS_{\bp_0, \bq_0}$. Proposition~\ref{prop strat}~\eqref{item strat2} implies that in order to prove that $\rep (A, \bd)$ is reducible, it is sufficient to find $\bp_1 \in \cP_{m_0} (d_0)$ and $\bq_1 \in \cP_{m_1} (d_1)$ such that $(\bp_1, \bq_1)$ is not maximal and
\[
\dim \cS_{\bp_1, \bq_1} \geq \dim \cS_{\max}.
\]
For $\bp \in \cP_{m_0} (d_0)$ and $\bq \in \cP_{m_1} (d_1)$, let $c_{\bp, \bq}$ be the codimension of $\bbA_{\bp, \bq}$ inside $(\Hom_K (K^{d_1}, K^{d_0}))^h$. Using Proposition~\ref{prop strat}~\eqref{item strat1} we get
\begin{multline} \label{eq diff}
\dim \cS_{\bp_1, \bq_1} - \dim \cS_{\max} = (c_{\bp_0, \bq_0} - c_{\bp_1, \bq_1})
\\
+ (\dim \cO_{\bp_1} - \dim \cO_{\bp_0}) + (\dim \cO_{\bq_1} - \dim \cO_{\bq_0}).
\end{multline}
In all the situations we apply this argument we will have (up to duality) that $\bq_1 = \bq_0$, $\bp_0 = (p)$ and $\bp_1 = (p - 1, 1)$ with $p \geq 2$. If this is the case, then Lemma~\ref{Lem orbits} implies that $\dim \cS_{\bp_1, \bq_1} \geq \dim \cS_{\max}$ if and only if $c_{\bp_0, \bq_0} - c_{\bp_1, \bq_1} \geq 2$. We formulate this result.

\begin{Cor} \label{cor irr}
Let $(\bp, \bq)$ be a maximal pair of partitions for a dimension vector $\bd$, such that $\bp = (p)$ for some $p \geq 2$. If
\[
c_{\bp, \bq} \geq c_{(p - 1, 1), \bq} + 2,
\]
then $\rep (A, \bd)$ is reducible. \qed
\end{Cor}

We explain now how to calculate $c_{\bp, \bq}$. Observe that $c_{\bp, \bq}$ is the rank of the system of linear equations determined by the conditions
\[
\rho_i (J_\bp, J_\bq, M_1, \ldots, M_h) = 0, \qquad 1 \leq i \leq k,
\]
with $(\rho_1, \ldots, \rho_k) = \cJ$ and $M_1$, \ldots, $M_h$ are matrices of indeterminates. If $\rho$ is a relation, then we denote by $c_{\bp, \bq} (\rho)$ the rank of the system of linear equations determined by the condition
\[
\rho (J_\bp, J_\bq, M_1, \ldots, M_h) = 0.
\]
As a first observation we get the following, which is obvious in the light of the preceding remarks.

\begin{Lem} \label{lem c part}
Assume $\cI = (\rho_1, \ldots, \rho_k)$ and there exists a partition $I_1$, \ldots, $I_k$ of $\{ 1, \ldots, h \}$ such that, for each $1 \leq i \leq k$,
\[
\rho_i = \sum_{j \in I_i} \sum_{\substack{0 \leq a < m_0 \\ 0 \leq b < m_1 \\ a + b \geq 1}} \lambda^{i, j}_{a, b} \vep_0^a \alpha_j \vep_1^b,
\]
i.e., different relations $\rho_i$ involve different arrows $\alpha_j$. Then
\[
c_{\bp, \bq} = \sum_{i = 1}^k c_{\bp, \bq} (\rho_i).
\]
\end{Lem}

Observe that $c_{\bp, \bq}$ behaves additively, i.e., if $\bp = (p_1, \ldots, p_l)$ and $\bq = (q_1, \ldots, q_k)$, then
\[
c_{\bp, \bq} = \sum_{i = 1}^l \sum_{j = 1}^k c_{(p_i), (q_j)}.
\]
Consequently it is sufficient to study partitions consisting of only one part. We have the following rules, where by a linear combination of arrows we mean a linear combination of arrows in $Q (h)$ starting at $1$ and ending at $0$.

\begin{Prop} \label{prop c formulas}
Let $1 \leq p \leq m_0$, $1 \leq q \leq m_1$, and $\alpha$ be a nonzero linear combinations of arrows.
\begin{enumerate}

\item \label{item c formula2}
If $l \leq p$, then
\[
c_{(p), (1)} (\vep_0^l \alpha) = p - l.
\]

\item \label{item c formula3}
If $m_1 \geq 2$ and $1 \leq l < p$, then
\[
c_{(p), (2)} (\vep_0^l \alpha + \vep_0^{l - 1} \alpha \vep_1) = 2 (p - l).
\]

\item \label{item c formula4}
If $q \leq p$, then
\[
c_{(p), (q)} (\vep_0 \alpha + \alpha \vep_1) = q (p - 1).
\]

\item
If $q \leq p$ and $\beta$ is a linear combination of arrows such that $\alpha$ and $\beta$ are linearly independent, then
\[
c_{(p), (q)} (\vep_0 \alpha + \beta \vep_1) = p q - 1.
\]

\item
If $2 \leq q \leq p$ and $\beta$ is a linear combination of arrows such that $\alpha$ and $\beta$ are linearly independent, then
\[
c_{(p), (q)} (\vep_0 \alpha + \alpha \vep_1 + \beta \vep_1^{q - 1}) = q (p - 1) + 1.
\]

\item
If $l \leq p, q$, then
\[
c_{(p), (q)} \Bigl( \sum_{i = q - l}^{q - 1} \vep_0^{p + q - l - i - 1} \alpha \vep_1^i \Bigr) = l.
\]

\item
If $p, q \geq 2$, $\gamma$ is a linear combination of arrows and $\lambda \in K$, then
\[
c_{(p), (q)} (\vep_0^{p - 1} \alpha \vep_1^{q - 2} + \lambda \vep_0^{p - 2} \alpha \vep_1^{q - 1} + \vep_0^{p - 1} \gamma \vep_1^{q - 1}) = 2.
\]

\item
If $p, q \geq 2$ and $\beta$ and $\gamma$ are linear combinations of arrows such that $\alpha$ and $\beta$ are linearly independent, then
\[
c_{(p), (q)} (\vep_0^{p - 1} \alpha \vep_1^{q - 2} + \vep_0^{p - 2} \beta \vep_1^{q - 1} + \vep_0^{p - 1} \gamma \vep_1^{q - 1}) = 3.
\]

\item
If $p, q \geq 3$, $\gamma'$ and $\gamma''$ are linear combinations of arrows and $\lambda \neq 1$, then
\[
c_{(p), (q)} (\vep_0^{p - 1} \alpha \vep_1^{q - 3} + \vep_0^{p - 2} \alpha \vep_1^{q - 2} + \lambda \vep_0^{p - 3} \alpha \vep_1^{q - 1} + \vep_0^{p - 2} \gamma' \vep_1^{q - 1} + \vep_0^{p - 1} \gamma'' \vep_1^{q - 1}) = 4.
\]


\item
If $3 \leq l \leq p, q$ and $\beta$ and $\gamma$ are linear combinations of arrows such that $\alpha$ and $\beta$ are linearly independent, then
\[
c_{(p), (q)} \Bigl( \sum_{i = q - l}^{q - 1} \vep_0^{p + q - l - i - 1} \alpha \vep_1^i  + \vep_0^{p - 2} \beta \vep_1^{q - 1} + \vep_0^{p - 1} \gamma \vep_1^{q - 1} \Bigr) = l + 1.
\]

\item
If $4 \leq l \leq p, q$, $\gamma'$ and $\gamma''$ are linear combinations of arrows and $\lambda \neq 0$, then
\[
c_{(p), (q)} \Bigl( \sum_{i = q - l}^{q - 1} \vep_0^{p + q - l - i - 1} \alpha \vep_1^i  + \lambda \vep_0^{p - 3} \alpha \vep_1^{q - 1} + \vep_0^{p - 2} \gamma' \vep_1^{q - 1} + \vep_0^{p - 1} \gamma'' \vep_1^{q - 1} \Bigr) = l + 1.
\]

\end{enumerate}
\end{Prop}

\begin{proof}
We only prove the last formula, which seems to be the most complicated one. The remaining ones are proved similarly. If $M_1, \ldots, M_h \in \Hom_K (K^q, K^p)$, then we write $M (\alpha)$ in a matrix form
\[
M (\alpha) =
\begin{bmatrix}
M_{p, 1} & \cdots & M_{p, q}
\\
\vdots & & \vdots
\\
M_{1, 1} & \cdots & M_{1, q}
\end{bmatrix}
\]
with unusual indexing. Similarly, we write
\[
M (\gamma') =
\begin{bmatrix}
M_{p, 1}' & \cdots & M_{p, q}'
\\
\vdots & & \vdots
\\
M_{1, 1}' & \cdots & M_{1, q}'
\end{bmatrix}
\qquad \text{and} \qquad
M (\gamma'') =
\begin{bmatrix}
M_{p, 1}'' & \cdots & M_{p, q}''
\\
\vdots & & \vdots
\\
M_{1, 1}'' & \cdots & M_{1, q}''
\end{bmatrix}.
\]
If $\rho := \sum_{i = q - l}^{q - 1} \vep_0^{p + q - l - i - 1} \alpha \vep_1^i  + \lambda \vep_0^{p - 3} \alpha \vep_1^{q - 1} + \vep_0^{p - 2} \gamma' \vep_1^{q - 1} + \vep_0^{p - 1} \gamma'' \vep_1^{q - 1}$, then
\[
\rho (J_{(p)}, J_{(q)}, M_1, \ldots, M_h) =
\begin{bmatrix}
0 & \cdots & 0 & x_2 & \cdots & x_{l - 2} & x_{l - 1} & x_l & x_{l + 1}
\\
0 & \cdots & \cdots & 0 & x_2 & \cdots & x_{l - 2} & x_{l - 1} & x_l'
\\
0 & \cdots & \cdots & \cdots & 0 & x_2 & \cdots & x_{l - 2} & x_{l - 1}'
\\
0 & \cdots & \cdots & \cdots & \cdots & 0 & x_2 & \cdots & x_{l - 2}
\\
\vdots & & & & & & \ddots & \ddots & \vdots
\\
\vdots & & & & & & & \ddots & x_2
\\
\vdots & & & & & & & & 0
\\
\vdots & & & & & & & & \vdots
\\
0 & \cdots & \cdots & \cdots & \cdots & \cdots & \cdots & \cdots & 0
\end{bmatrix},
\]
where
\begin{gather*}
x_i := \sum_{j = 1}^{l - 1} M_{j, l - j}, \qquad 2 \leq i \leq l,
\\
x_{l - 1}' := x_{l - 1} + \lambda x_2, \qquad x_l' := x_l + \lambda M_{2, 1} + M_{1, 1}'
\\
\intertext{and}
x_{l + 1} := x_l + \lambda M_{3, 1} + M_{2, 1}' + M_{3, 1}''.
\end{gather*}
Now one sees that $x_2$, \ldots, $x_{l + 1}$ and $x_l'$ are linearly independent.
\end{proof}

\begin{Rem}
The formulas~\eqref{item c formula2}, \eqref{item c formula3} and~\eqref{item c formula4} can be unified and generalized to the following formula: if $l < p$ and  $q \leq p$, then
\[
c_{(p), (q)} \Bigl( \sum_{i = 0}^{\min \{ l, q - 1 \}} \vep_0^{l - i} \alpha \vep_1^i \Bigr) = q (p - l).
\]
However in this case the analysis is more complicated than the one presented above. Moreover, we do not need this statement. As a result we decided not to include this claim in Proposition~\ref{prop c formulas}.
\end{Rem}


\section{Proof of Theorem~\ref{thm:intro3}} \label{sect proof one}


\subsection{Notation} \label{subsect not}
We present some notation we use throughout this section. Let $h \geq 1$ and $m_0, m_1 \geq 1$ be such that $m_0 + m_1 \geq 3$. Put $Q := Q (h)$ and $B := A' (h, m_0, m_1)$. Fix relations $\rho_1$, \ldots, $\rho_k$ in $B$ such that, for each $1 \leq i \leq k$, $s (\rho_i) = 1$ and $t (\rho_i) = 0$. Let $\cJ$ be the ideal generated $\rho_1$, \ldots, $\rho_k$, and let $A := B / \cJ$. Observe that
\[
\cJ = \sum_{i = 1}^k B_0 \rho_i B_1,
\]
and $A_0 = B_0$ and $A_1 = B_1$ (we again remind that $A_i := e_i A e_i$ and $B_i := e_i B e_i$, for $i = 0, 1$). We assume that $k$ is minimal in the sense that, if $\rho_1'$, \ldots, $\rho_l'$ are relations such that
\[
\sum_{j = 1}^l B_0 \rho_j' B_1 = \cJ,
\]
then $l \geq k$. For each $1 \leq i \leq k$, write
\[
\rho_i = \sum_{j = 1}^h \sum_{\substack{0 \leq a < m_0 \\ 0 \leq b < m_1 \\ a + b \geq 1}} \lambda^{i, j}_{a, b} \vep_0^a \alpha_j \vep_1^b,
\]
for some $\lambda^{i, j}_{a, b} \in K$.

Our aim is to prove that if $\cJ \neq 0$ and $A$ is geometrically irreducible, then we may assume $m_0 = m_1$, $k = 1$, and either $\rho_1 = \rho^{(1)}$ or $\rho_1 = \rho^{(m_0 - 1)}$, which will finish the proof of Theorem~\ref{thm:intro3}.

\subsection{Each relation contains monomials of degree $(a, 0)$ and $(0, b)$}
Before we can proceed, we need some module theoretic considerations.

\begin{Lem}\label{lem:a0prep}
Let $\Lambda$ be a $K$-algebra, and let $N_1 \subseteq N \subseteq M$ be $\Lambda$-modules such that the following hold:
\begin{enumerate}

\item
$N$ is a direct summand of $M$;

\item
$N_1 + \rad (M) = N + \rad (M)$.

\end{enumerate}
Then $N_1 = N$.
\end{Lem}

\begin{proof}
We have $M = N \oplus N'$ for some submodule $N'$ of $M$. It follows that $\rad (M) = \rad (N) \oplus \rad (N')$, and therefore $N + \rad (M) = N \oplus \rad (N')$ and $N_1 + \rad (M) = (N_1 + \rad (N)) \oplus \rad (N')$. Since $N_1 + \rad (M) = N + \rad (M)$, $N = N_1 + \rad (N)$. This implies $N_1 = N$ by the Nakayama Lemma.
\end{proof}

We prove the following.

\begin{Lem} \label{lem:a0}
If $A$ is geometrically irreducible, then for each $1 \leq i \leq k$, there exist $a$, $j_0$ and $b$, $j_1$ such that $\lambda^{i, j_0}_{a, 0} \neq 0$ and $\lambda^{i, j_1}_{0, b} \neq 0$.
\end{Lem}

\begin{proof}
Let $\Lambda := B_0$, $M := e_0 B e_1$ and $N := \cJ$. Note that $N$ is a $\Lambda$-submodule of $M$. Moreover, $N' := M / N = e_0 A e_1$. Since $A$ is geometrically irreducible, we know from Proposition~\ref{prop known}~(3) that $N'$ is a free $\Lambda$-module. It follows that $N$ is a direct summand of $M$.

Let
\begin{gather*}
I := \{ 1 \leq i \leq k \mid \text{there exist $b$ and $j$ such that $\lambda^{i, j}_{0, b} \neq 0$} \}
\\
\intertext{and}
N_1 := \sum_{i \in I} B_0 \rho_i B_1.
\end{gather*}
We claim that
\[
N_1 + \rad_\Lambda (M) = N + \rad_\Lambda (M).
\]
Obviously, we have $N_1 + \rad_\Lambda (M) \subseteq N + \rad_\Lambda (M)$. For the converse inclusion, observe that $\rho_i \in \rad_\Lambda (M)$, for each $i \notin I$. Consequently, $\rho_i \in N_1 + \rad_\Lambda (M)$, for each $1 \leq i \leq k$, hence $N \subseteq N_1 + \rad_\Lambda (M)$. Obviously, $\rad_\Lambda (M) \subseteq N_1 + \rad_\Lambda (M)$, hence the equality follows.

Now we apply Lemma~\ref{lem:a0prep} and get $N_1 = N$. Then minimality of $k$ implies that $I = \{ 1, \ldots, k \}$, i.e., for each $1 \leq i \leq k$ there exist $b$ and $j$ such that $\lambda^{i, j}_{0, b} \neq 0$. Dually we prove that for each $1 \leq i \leq k$ there exist $a$ and $j$ such that $\lambda^{i, j}_{a, 0} \neq 0$.
\end{proof}

\subsection{Reduction to one relation}
The aim of this section is to prove the following result.

\begin{Lem}\label{lem:1relation}
If $A$ is geometrically irreducible, then $k = 1$.
\end{Lem}

\begin{proof}
Assume $k > 1$. Let $a_1$ be the minimal $a$ such that $\lambda^{i, j}_{a, 0} \neq 0$ for some $i, j \geq 1$. Lemma~\ref{lem:a0} ensures that such an $a_1$ exists. Without loss of generality we assume that $i = 1 = j$.

Thus $\rho_1$ is of the form
\[
\rho_1 = \vep_0^{a_1} \sum_{j = 1}^h \sum_{a_1 \leq a < m_0} \lambda^{1, j}_{a, 0} \vep_0^{a - a_1} \alpha_j + \sum_{j = 1}^h \sum_{\substack{0 \leq a < m_0 \\ 1 \leq b < m_1}} \lambda^{1, j}_{a, b} \vep_0^a \alpha_j \vep_1^b.
\]
Rescaling by
\[
\alpha_1 \leftarrow \sum_{j = 1}^h \sum_{a_1 \leq a < m_0} \lambda^{1, j}_{a, 0} \vep_0^{a - a_1} \alpha_j
\]
we can assume that $\rho_1$ is of the form
\[
\rho_1 = \vep_0^{a_1} \alpha_1  + \sum_{j = 1}^h \sum_{\substack{0 \leq a < m_0 \\ 1 \leq b < m_1}} \lambda^{1, j}_{a, b} \vep_0^a \alpha_j \vep_1^b
\]
(in the paper we use the following notation: $\alpha \leftarrow \omega$ means that $\omega$ becomes the arrow $\alpha$, while the remaining arrows do not change; note that the coefficients $\lambda_{a, b}^{1, j}$ may change). Moreover, for $2 \leq i \le k$, we can now assume that $\lambda_{a, 0}^{i, 1} = 0$ for any $a$. (This can be achieved by subtracting appropriate linear combinations of $\vep_0^{a - a_1} \rho_1$ from $\rho_i$.) On the other hand, if $\lambda_{a, 0}^{i, j} \neq 0$, for some $i, j \geq 2$ and $a$, then $a \geq a_1$.

Next, let $a_2$ be the minimal $a$ such that $\lambda^{i, j}_{a,0} \neq 0$ for some $i, j \geq 2$. Again existence of $a_2$ (in particular, that $h \geq 2$) follows from Lemma~\ref{lem:a0}. Observe that $a_2 \geq a_1$. Without loss of generality we assume that $i = 2 = j$. After the arrow transformation
\[
\alpha_2 \leftarrow \sum_{j = 2}^h \sum_{a_0 \leq a < m_0} \lambda^{2, j}_{a,0} \vep_0^{a - a_2} \alpha_j
\]
we can obtain that $\rho_1$ and $\rho_2$ are of the form
\begin{gather*}
\rho_1 = \vep_0^{a_1} \alpha_1  + \sum_{j = 1}^h \sum_{\substack{0 \leq a < m_0 \\ 1 \leq b < m_1}} \lambda^{1, j}_{a, b} \vep_0^a \alpha_j \vep_1^b.
\\
\intertext{and}
\rho_2 = \vep_0^{a_2} \alpha_1  + \sum_{j = 1}^h \sum_{\substack{0 \leq a < m_0 \\ 1 \leq b < m_1}} \lambda^{2, j}_{a, b} \vep_0^a \alpha_j \vep_1^b.
\end{gather*}
Moreover, for $3 \leq i \le k$, we can again assume that $\lambda_{a, 0}^{i, 1} = 0 = \lambda_{a, 0}^{i, 2}$ for any $a$, while if $\lambda_{a, 0}^{i, j} \neq 0$, for some $i, j \geq 3$ and $a$, then $a \geq a_2$.

Continuing inductively in this way, we can assume that, for each $1 \leq i \leq k$,
\[
\rho_i = \vep_0^{a_i} \alpha_i + \sum_{j = 1}^h \sum_{\substack{0 \leq a < m_0 \\ 1 \leq b < m_1}} \lambda^{i, j}_{a, b} \vep_0^a \alpha_j \vep_1^b,
\]
with $0 < a_1 \leq a_2 \leq \cdots \leq a_k < m_0$. In particular, this means $k \leq h$.

Let $\bd = (a_2 + 1, 1)$. We show that $\rep (A, \bd)$ is reducible. Observe that $\rep (A, \bd)$ is described by the relations
\[
\vep_0^{a_1} \alpha_1, \; \vep_0^{a_2} \alpha_2, \; \cdots, \; \vep_0^{a_k} \alpha_k
\]
(obviously, if $i > 2$ and $a_i > a_2$, then the relation $\vep_0^{a_i} \alpha_i$ is redundant). The maximal partitions in $\cP_{m_0} (a_2 + 1)$ and $\cP_{m_1} (1)$ are $\bp_0 := (a_2 + 1)$ and $\bq_0 := (1)$, respectively. Using Lemma~\ref{lem c part} we get
\[
c_{\bp_0, \bq_0} = \sum_{i = 1}^k c_{\bp_0, \bq_0} (\vep_0^{a_i} \alpha_i).
\]
Applying Proposition~\ref{prop c formulas}~\eqref{item c formula2} for $p = a_2 + 1$, $q = 1$ and $l = a_1$ ($l = a_2$, respectively), we get
\[
c_{\bp_0, \bq_0} (\vep_0^{a_1} \alpha_1) = a_2 + 1 - a_1 \qquad \text{and} \qquad c_{\bp_0, \bq_0} (\vep_0^{a_2} \alpha_2) = 1,
\]
hence
\[
c_{\bp_0, \bq_0} \geq a_2 + 2 - a_1.
\]
Similarly, if we put $\bp_1 := (a_2, 1)$ and $\bq_1 := \bq_0$, then $\cS_{\bp_1, \bq_1}$ is described by the relation $\vep_0^{a_1} \alpha_1$, hence we get
\[
c_{\bp_1, \bq_1} = c_{\bp_1, \bq_1} (\vep_0^{a_1} \alpha_1) = c_{(a_2), (1)} (\vep_0^{a_1} \alpha_1) = a_2 - a_1,
\]
which contradicts irreducibility according to Corollary~\ref{cor irr}.
\end{proof}

\subsection{The one relation case}
Our aim in this subsection is to prove the following result.

\begin{Prop} \label{prop onerel}
If $\cJ \neq 0$ and $A$ is geometrically irreducible, then
\[
A \cong A (h, m_0, m_0, 1) \qquad \text{or} \qquad A \cong A (h, m_0, m_0, m_0 - 1).
\]
In other words, $m_0 = m_1$, $k = 1$, and we may assume either $\rho_1 = \rho^{(1)}$ or $\rho_1 = \rho^{(m_0 - 1)}$.
\end{Prop}

The proof of Proposition~\ref{prop onerel} will be done in steps. We already know $k = 1$. Hence we may simplify our notation and put $\rho := \rho_1$.
Without loss of generality we may assume $m_0 \geq m_1$. Then in particular $m_0 \geq 2$ (recall that we assume $m_0 + m_1 \geq 3$). Throughout the proof we use freely Corollary~\ref{cor irr} and Proposition~\ref{prop c formulas} without explicitly referring to them.

\begin{step} \label{step1}
There exists $0 < n < m_0$ such that we may assume that, for each $0 \leq p \leq \min \{ n, m_1 - 1 \}$,
\[
\rho = \sum_{i = 0}^p \varepsilon_0^{n - i} \alpha_1 \varepsilon_1^i + \omega \varepsilon_1^{p + 1},
\]
for some $\omega$. Furthermore, if $n > 1$, then $m_0 = n + 1$.

In particular, if $m_1 \leq n + 1$, then we get
\[
\rho = \sum_{i = 0}^{m_1 - 1} \varepsilon_0^{n - i} \alpha_1 \varepsilon_1^i.
\]
On the other hand, if $m_1 > n + 1$, then
\[
\rho = \sum_{i = 0}^n \varepsilon_0^{n - i} \alpha_1 \varepsilon_1^i + \omega \varepsilon_1^{n + 1}.
\]
\end{step}

\begin{proof}[Proof of Step~\ref{step1}]
The proof will be inductive with respect to $p$.

\underline{\textit{The base case}}

Write
\[
\rho = \sum_{j = 1}^h \sum_{\substack{0 \leq a < m_0 \\ 0 \leq b < m_1 \\ a + b \geq 1}} \lambda^j_{a, b} \vep_0^a \alpha_j \vep_1^b.
\]
Let
\[
n := \min \{ a \mid \text{$\lambda_{a, 0}^j \neq 0$ for some $j$} \}.
\]
Note that $n$ is well-defined (by Lemma~\ref{lem:a0}) and $0 < n < m_0$. We may assume $\lambda_{n, 0}^1 \neq 0$ and make the substitution
\[
\alpha_1 \leftarrow \sum_{a, j} \lambda_{a, 0}^j \varepsilon_0^{a - n} \alpha_j.
\]
As a result we get $\rho$ of the form
\[
\varepsilon_0^n \alpha_1 + \omega \varepsilon_1,
\]
for some $\omega$, which shows the claim for $p = 0$. Note that this finishes the proof of Step~\ref{step1} if $m_1 = 1$.

\underline{\textit{The inductive step}}

Let $0 < p \leq \min \{ n, m_1 - 1 \}$ and assume
\[
\rho = \sum_{i = 0}^{p - 1} \varepsilon_0^{n - i} \alpha_1 \varepsilon_1^i + \omega \varepsilon_1^p.
\]
Write
\[
\omega = \sum_{a, j} \lambda_a^j \varepsilon_0^a \alpha_j + \omega' \varepsilon_1,
\]
and put
\[
\omega_q := \sum_j \lambda_q^j \alpha_j \qquad (q = 0, \ldots, n - 1)
\]
and
\[
\omega_n := \sum_{a \geq n, j} \lambda_a^j \varepsilon_0^{a - n} \alpha_j.
\]
In particular,
\[
\rho = \sum_{i = 0}^{p - 1} \varepsilon_0^{n - i} \alpha_1 \varepsilon_1^i + \sum_{q = 0}^n \varepsilon_0^q \omega_q \varepsilon_1^p + \omega' \varepsilon_1^{p + 1}.
\]

$1^\circ$ ($q = n$). We show we may assume $\omega_n = 0$.

We substitute
\[
\alpha_1 \leftarrow \alpha_1 + \omega_n \varepsilon_1^p,
\]
and get
\[
\rho = \sum_{i = 0}^{p - 1} \varepsilon_0^{n - i} \alpha_1 \varepsilon_1^i + \sum_{q = 0}^{n - 1} \varepsilon_0^q \omega_j \varepsilon_1^p + \omega'' \varepsilon_1^{p + 1},
\]
for some $\omega''$.

$2^\circ$ ($0 \leq q \leq n - p - 1$): We show that
\[
\omega_0 = \cdots = \omega_{n - p - 1} = 0
\]
(there is nothing to prove if $p = n$). If this is not the case, let
\[
l := \min \{ q \in \{ 0, \ldots, n - p - 1 \} \mid \omega_q \neq 0 \}.
\]
Take the dimension vector $(l + 2, p + 1)$. Observe that $l + 2 \leq n - p + 1$. Consequently, the variety is described by the relation
\[
\varepsilon_0^l \omega_l \varepsilon_0^p + \varepsilon_0^{l + 1} \omega_{l + 1} \varepsilon_0^p.
\]
The maximal pair of partitions is $((l + 2), (p + 1))$ and $c_{(l + 2), (p + 1)} = 2$. On the other hand, we get no conditions for the stratum $\cS_{(l + 2), (p, 1)}$, hence $c_{(l + 2), (p, 1)} = 0$, which contradicts irreducibility.

Thus we get
\[
\rho = \sum_{i = 0}^{p - 1} \varepsilon_0^{n - i} \alpha_1 \varepsilon_1^i + \sum_{q = n - p}^{n - 1} \varepsilon_0^q \omega_q \varepsilon_1^p + \omega'' \varepsilon_1^{p + 1}.
\]

$3^\circ$ ($q = n - p$): We first show that $\omega_{n - p} = \lambda_{n - p}^1 \alpha_1$ and $\lambda_{n - p}^1 \neq 0$.

Assume first $\omega_{n - p} = 0$ and consider the dimension vector $(n - p + 2, p + 1)$. Then the variety is described by the relation
\[
\varepsilon_0^{n - p + 1} \alpha_1 \varepsilon_1^{p - 1} + \varepsilon_0^{n - p + 1} \omega_{n - p + 1} \varepsilon_1^p.
\]
The maximal pair of partitions is $((n - p + 2), (p + 1))$ and $c_{(n - p + 2), (p + 1)} = 2$. On the other there are no conditions for the stratum $\cS_{(n - p + 1, 1), (p + 1)}$, hence $c_{(n - p + 1, 1), (p + 1)} = 0$, which contradicts irreducibility.

Next assume there exists $j \neq 1$ such that $\lambda_{n - p}^j \neq 0$ and again consider the dimension vector $(n - p + 2, p + 1)$. In this case the variety is described by the relation
\[
\varepsilon_0^{n - p + 1} \alpha_1 \varepsilon_1^{p - 1} + \varepsilon_0^{n - p} \omega_{n - p} \varepsilon_1^p + \varepsilon_0^{n - p + 1} \omega_{n - p + 1} \varepsilon_1^p.
\]
The maximal pair of partitions is $((n - p + 2), (p + 1))$ and $c_{(n - p + 2), (p + 1)} = 3$ (note that $\alpha_1$ and $\omega_{n - p}$ are linearly independent). Let $(\bp, \bq) := ((n - p + 1, 1), (p + 1))$, if $n > p$, and $(\bp, \bq) := ((n - p + 2), (p, 1))$, if $n = p > 1$. Then $\cS_{\bp, \bq}$ is described by the condition $\varepsilon_0^{n - p} \omega_{n - p} \varepsilon_1^p$ in the former case and by the condition $\varepsilon_0^{n - p + 1} \alpha_1 \varepsilon_1^{p - 1}$ in the latter case. In both cases we get $c_{\bp, \bq} = 1$, which contradicts irreducibility.

Thus we only need to deal with the case $n = p = 1$. In this case we consider the dimension vector $(2 m_0, 2)$. Then the variety is described by the relation
$\varepsilon_0 \alpha + \omega_0 \varepsilon_1$. The maximal pair of partitions is $((m_0, m_0), (2))$ and $c_{(m_0, m_0), (2)} = 2 (2 m_0 - 1) = 4 m_0 - 2$. On the other hand $c_{((m_0, m_0), (1, 1))} = 4 (m_0 - 1) = 4 m_0 - 4$, since $\cS_{(m_0, m_0), (1, 1)}$ is described by the condition $\varepsilon_0 \alpha$, which contradicts irreducibility.

Now we show we may assume $\lambda_{n - p}^1 = 1$. If $p = 1$, then we substitute
\[
\varepsilon_1 \leftarrow \lambda_{n - 1}^1 \varepsilon_1.
\]

\underline{\textit{Intermezzo}}

Observe we have already proved our inductive claim for $p = 1$. In particular, this finishes the proof of Step~\ref{step1} if $n = 1$. We additionally show that if $n > 1$, then $m_0 = n + 1$, which we use in the rest of the proof of Step~\ref{step1}. Indeed, if this is not the case, then consider the dimension vector $(n + 2, 2)$. The variety is described by the relation $\varepsilon_0^n \alpha_1 + \varepsilon_0^{n - 1} \alpha_1 \varepsilon_1$. The maximal pair of partitions is $((n + 2), (2))$ and $c_{((n + 2), (2))} = 4$. Moreover, $c_{(n + 1, 1), (2)} = 2$, which contradicts irreducibility.

\underline{\textit{Inductive step (continuation of $3^\circ$)}}

Assume $p > 1$ and $\lambda_{n - p}^1 \neq 1$. If either $p < n$ or $p > 2$, then we consider the dimension vector $(n - p + 3, p + 1)$. The variety is described by the relation
\begin{multline*}
\varepsilon_0^{n - p + 2} \alpha_1 \varepsilon_1^{p - 2} + \varepsilon_0^{n - p + 1} \alpha_1 \varepsilon_1^{p - 1} + \lambda_{n - p}^1 \varepsilon_0^{n - p} \alpha_1 \varepsilon_1^p
\\
+ \varepsilon_0^{n - p + 1} \omega_{n - p + 1} \varepsilon_1^p + \varepsilon_0^{n - p + 2} \omega_{n - p + 2} \varepsilon_1^p
\end{multline*}
(note that $1^\circ$ implies $\omega_{n - p + 2} = 0$ if $p = 2$). The maximal pair of partitions is $((n - p + 3), (p + 1))$ and $c_{(n - p + 3), (p + 1)} = 4$. Let $(\bp, \bq) := ((n - p + 2, 1), (p + 1))$, if $p < n$, and $(\bp, \bq) := ((n - p + 3), (p, 1))$, if $n = p > 2$. Then $c_{\bp, \bq} = 2$, which contradicts irreducibility.

Finally, assume $p = n = 2$ and $\lambda_{n - p}^1 \neq 1$. We consider the dimension vector $(2 m_0, p + 1) = (6, 3)$ (recall we proved $m_0 = n + 1$, provided $n > 1$). The variety is described by the relation
\[
\varepsilon_0^2 \alpha_1 + \varepsilon_0 \alpha_1 \varepsilon_1 + \lambda_0^1 \alpha_1 \varepsilon_1^2 + \varepsilon_0 \omega_1 \varepsilon_1^2
\]
and the maximal pair of partitions is $((3, 3), (3))$. Then $c_{(3, 3), (3)} = 2 \cdot 4 = 8$ and $c_{(3, 3), (2, 1)} = 2 \cdot (2 + 1) = 6$, hence we again get a contradiction.

Thus we get
\[
\rho = \sum_{i = 0}^p \varepsilon_0^{n - i} \alpha_1 \varepsilon_1^i + \sum_{q = n - p + 1}^{n - 1} \varepsilon_0^q \omega_q \varepsilon_1^p + \omega'' \varepsilon_1^{p + 1}.
\]

$4^\circ$ ($n - p + 1 \leq q \leq n - 2$): We prove that
\[
\omega_{n - p + 1} = \cdots = \omega_{n - 2} = 0
\]
(obviously, there is nothing to prove if $p \leq 2$). Assume this is not the case and put
\[
l := \min \{ q \in \{ n - p + 1, \ldots, n - 2 \} \mid \omega_q \neq 0 \}.
\]
If there exists $j \neq 1$ such that $\lambda_l^j \neq 0$, then we consider the dimension vector $(l + 2, p + 1)$. The variety is described by the relation
\[
\sum_{i = n - l - 1}^p \varepsilon_0^{n - i} \alpha_1 \varepsilon_1^i + \varepsilon_0^l \omega_l \varepsilon_1^p + \varepsilon_0^{l + 1} \omega_{l + 1} \varepsilon_1^p
\]
and the maximal pair of partitions is $((l + 2), (p + 1))$. Note that $c_{(l + 2), (p + 1)} = p + l + 3 - n$, while $c_{(l + 2), (p, 1)} = p + l + 1 - n$, which contradicts geometric irreducibility of $A$. Thus $\omega_l = \lambda_l^1 \alpha_1$.

If $\lambda_l^1 \neq 0$ and $l < n - 2$, then we consider the dimension vector $(l + 3, p + 1)$. The variety is described by the relation
\[
\sum_{i = n - l - 2}^p \varepsilon_0^{n - i} \alpha_1 \varepsilon_1^i + \lambda_l^1 \varepsilon_0^l \alpha_1 \varepsilon_1^p + \varepsilon_0^{l + 1} \omega_{l + 1} \varepsilon_1^p + \varepsilon_0^{l + 2} \omega_{l + 2} \varepsilon_1^p
\]
and the maximal pair of partitions is $((l + 3), (p + 1))$. We have $c_{(l + 3), (p + 1)} = p + l + 4 - n$ and $c_{(l + 3), (p, 1)} = p + l + 2 - n$, which contradicts irreducibility.

Finally, if $\lambda_l^1 \neq 0$ and $l = n - 2$, then we consider the dimension vector $(2 m_0, p + 1) = (2 n + 2, p + 1)$. In this case the variety is described by the relation
\[
\sum_{i = 0}^p \varepsilon_0^{n - i} \alpha_1 \varepsilon_1^i + \lambda_l^1 \varepsilon_0^{n - 2} \alpha_1 \varepsilon_1^p + \varepsilon_0^{n - 1} \omega_{n - 1} \varepsilon_1^p
\]
and the maximal pair of partitions is $((n + 1, n + 1), (p + 1))$. We have $c_{(n + 1, n + 1), (p + 1)} = 2 (p + 2) = 2 p + 4$ and $c_{(n + 1, n + 1), (p, 1)} = 2 (p + 1) = 2 p + 2$, hence the claim follows.

Thus if $p \geq 2$, we get
\[
\rho = \sum_{i = 0}^p \varepsilon_0^{n - i} \alpha_1 \varepsilon_1^i + \varepsilon_0^{n - 1} \omega_{n - 1} \varepsilon_1^p + \omega'' \varepsilon_1^{p + 1}.
\]

$5^\circ$ ($q = n - 1$): If $p \geq 2$, we may assume $\omega_{n - 1} = 0$.

First assume there exists $j \neq 1$ such that $\lambda_{n - 1}^j \neq 0$ and consider the dimension vector $(2 m_0, p + 1) = (2 n + 2, p + 1)$. The variety is described be the relation
\[
\sum_{i = 0}^p \varepsilon_0^{n - i} \alpha_1 \varepsilon_1^i + \varepsilon_0^{n - 1} \omega_{n - 1} \varepsilon_1^p
\]
and the maximal pair of partitions is $((n + 1, n + 1), (p + 1))$. We get $c_{(n + 1, n + 1), (p + 1)} = 2 (p + 2) = 2 p + 4$ and $c_{(n + 1, n + 1), (p, 1)} = 2 (p + 1) = 2 p + 2$. Thus irreducibility implies $\omega_{n - 1} = \lambda_{n - 1}^1 \alpha_1$ and  we obtain the claim by substituting
\[
\varepsilon_1 \leftarrow \varepsilon_1 + \lambda_{n - 1}^1 \varepsilon_1^p.
\]
(we use here that $p \geq 2$).
\end{proof}

\begin{step} \label{step2}
If $n > 1$, then $m_1 = n + 1$.
\end{step}

Note that Step~\ref{step2} finishes the proof if $n > 1$.

\begin{proof}[Proof of Step~\ref{step2}]
We know $m_1 \leq m_0 = n + 1$. If $m_1 < m_0$, then consider the dimension vector $(m_0 - m_1 + 1, 2 m_1)$. Then the variety is described by the relation
\[
\varepsilon_0^{m_0 - m_1} \alpha_1 \varepsilon_1^{m_1 - 1}
\]
and the maximal pair of partitions is $((m_0 - m_1 + 1), (m_1, m_1))$. Since $c_{(m_0 - m_1 + 1), (m_1, m_1)} = 2 \cdot 1 = 2$ and $c_{(m_0 - m_1, 1), (m_1, m_1)} = 0$, we get a contradiction to irreducibility.
\end{proof}

\begin{step} \label{step3}
If $n = 1$, then we may assume that, for each $2 \leq p \leq m_1$,
\[
\rho = \varepsilon_0 \alpha_1 + \alpha_1 \varepsilon_1 + \omega \varepsilon_1^p,
\]
for some $\omega$. In particular, for $p = m_1$, we get
\[
\rho = \varepsilon_0 \alpha_1 + \alpha_1 \varepsilon_1.
\]
\end{step}

\begin{proof}[Proof of Step~\ref{step3}]
The proof will be inductive with respect to $p$.

\underline{\textit{The base case}}

If $p = 2$, then the claim follows from Step~\ref{step1}.

\underline{\textit{The inductive step}}

Assume $3 \leq p \leq m_1$ and
\[
\rho = \varepsilon_0 \alpha_1 + \alpha_1 \varepsilon_1 + \omega \varepsilon_1^{p - 1},
\]
for some $\omega$. Write
\[
\omega = \sum_{a, j} \lambda_a^j \varepsilon_0^a \alpha_j \varepsilon_1 + \omega' \varepsilon_1,
\]
and put
\[
\omega_0 := \sum_j \lambda_0^j \alpha_j \qquad \text{and} \qquad \omega_1 := \sum_{a \geq 1, j} \lambda_a^j \varepsilon_0^{a - 1} \alpha_j,
\]
thus
\[
\rho = \vep_0 \alpha_1 + \alpha_1 \vep_1 + \omega_0 \vep_1^{p - 1} + \vep_0 \omega_1 \vep_1^{p - 1} + \omega' \vep_1^p.
\]

$1^\circ$ ($q = 1$). If we substitute
\[
\alpha_1 \leftarrow \alpha_1 + \omega_1 \varepsilon_1^{p - 1},
\]
then we get
\[
\rho = \varepsilon_0 \alpha_1 + \alpha_1 \varepsilon_1 + \omega_0 \varepsilon_1^{p - 1} + \omega'' \varepsilon_1^p,
\]
for some $\omega''$.

$2^\circ$ ($q = 0$). Assume first there exists $j \neq 1$ such that $\lambda_a^j \neq 0$. If we consider the dimension vector $(2 m_0, p)$, then the variety is described by the relation
\[
\varepsilon_0 \alpha_1 + \alpha_1 \varepsilon_1 + \omega_0 \varepsilon_1^{p - 1}
\]
and the maximal pair of partitions is $((m_0, m_0), (p))$. We have $c_{(m_0, m_0), (p)} = 2 (p (m_0 - 1) + 1) = 2 p m_0 - 2 p + 2$ (note that $p \leq m_1 \leq m_0$), while $c_{(m_0, m_0), (p - 1, 1)} = 2 [(p - 1) (m_0 - 1) + (m_0 - 1)] = 2 p m_0 - 2 p$, which contradicts irreducibility.

Thus $\omega_0 = \lambda_0^1 \alpha_1$. Now we substitute
\[
\varepsilon_1 \leftarrow \varepsilon_1 + \lambda_0^1 \varepsilon_1^{p - 1}
\]
and we get
\[
\rho = \varepsilon_0 \alpha_1 + \alpha_1 \varepsilon_1 + \omega''' \varepsilon_1^p,
\]
for some $\omega'''$.
\end{proof}

\begin{step} \label{step4}
We have $m_0 = m_1$.
\end{step}

\begin{proof}
We know the claim holds if $n > 1$, hence assume $n = 1$. In particular,
\[
\rho = \varepsilon_0 \alpha_1 + \alpha_1 \varepsilon_1.
\]
If $m_1 < m_0$, then consider the dimension vector $(m_0, 2 m_1)$. Then the maximal pair of partitions is $((m_0), (m_1, m_1))$ and $c_{(m_0), (m_1, m_1)} = 2 m_1 (m_0 - 1) = 2 m_0 m_1 - 2 m_1$. On the other hand $c_{(m_0 - 1, 1), (m_1, m_1)} = 2 m_1 (m_0 - 2) + 2 (m_1 - 1) = 2 m_0 m_1 - 2 m_1 - 2$. This contradiction finishes the proof of Step~\ref{step4} and Proposition~\ref{prop onerel}.
\end{proof}

\subsection{Proof of Theorem~\ref{thm:intro3}}
We finish now Section~\ref{sect proof one} explaining how the above imply Theorem~\ref{thm:intro3}. Let $A$ be a geometrically irreducible algebra with exactly two simple modules. Corollary~\ref{coro quiver} implies $A$ is Morita equivalent to $A' (h,m_0, m_1) / \cJ$, for some $m_0, m_1 \geq 1$, $h \geq 0$, and an ideal $\cJ$. We may also assume $m_0$, $m_1$ and $h$ are minimal. If $\cJ = 0$, we are done, thus assume $\cJ \neq 0$. In this case $m_0 + m_1 \geq 3$ and $h \geq 1$, thus we are in setup of Section~\ref{sect proof one}, hence Proposition~\ref{prop onerel} implies the claim. \qed


\section{Proof of Theorem~\ref{thm:conj}} \label{sect proof two}


\subsection{Degree of a relation}
Let $Q$ be a quiver. For a path $\sigma = \alpha_1 \cdots \alpha_l$ we define its degree $\deg \sigma$ to be the number of indices $i$ such that $\alpha_i$ is not a loop. If $\rho = \sum_{i = 1}^k \lambda_i \sigma_i$ with $\lambda_1, \ldots, \lambda_k \neq 0$ and $\sigma_1$, \ldots, $\sigma_k$ pairwise different, then the degree $\deg \rho$ of $\rho$ is the minimum of $\deg \sigma_i$, $i = 1, \ldots, k$. Recall that by a shortcut we mean an arrow $\alpha$ such that there exists a path $\sigma$ of degree at least $2$ such that $s (\sigma) = s (\alpha)$ and $t (\sigma) = t (\alpha)$.

\begin{Lem} \label{lem:degree}
Let $A = K Q / \cI$ be a geometrically irreducible bound quiver algebra. If there are no shortcuts in $Q$, then $\cI$ is generated by relations of degree at most $1$.
\end{Lem}

\begin{proof}
Let $R_0$ and $R_1$ be the sets of relations of degree $0$ and $1$ in $\cI$, respectively. Our aim is to show that $\cI \subseteq \cJ :=(R_0 \cup R_1)$. Fix a relation $\rho = \sum_{i = 1}^k \lambda_i \sigma_i$ in $\cI$ and assume $\deg \rho \geq 2$. For a path $\sigma = \alpha_1 \cdots \alpha_n$ let $\supp (\sigma) := \{ t (\sigma), s (\alpha_1), \ldots, s (\alpha_n) \} \subseteq Q_0$. It is sufficient to show that $\rho_I \in \cJ$, for each $I \subseteq Q_0$, where
\[
\rho_I := \sum_{i : \supp (\sigma_i) = I} \lambda_i \sigma_i.
\]
In other words, we may assume $\supp (\sigma_1) = \cdots = \supp (\sigma_k) =: I$.

Let $B := A / A e A$, where $e := \sum_{i \not \in I} e_i$. Since there are nor shortcuts neither oriented cycles of positive degree in $Q$ (see Proposition~\ref{prop known} for the latter claim), $B = K Q' / \cI'$, where $Q'$ is a subquiver of the quiver
\[
\xymatrix{
0 \ar@(ul,ur)^{\vep_0}
&
1 \ar@(ul,ur)^{\vep_1} \ar@<-1ex>[l]_{\alpha_{11}} \ar@<1ex>^{\alpha_{h_11}}_{\cdots}[l]
&
2 \ar@(ul,ur)^{\vep_2} \ar@<-1ex>[l]_{\alpha_{12}} \ar@<1ex>^{\alpha_{h_22}}_{\cdots}[l]
&
\cdots \ar@<-1ex>[l]_{\alpha_{13}} \ar@<1ex>^{\alpha_{t3}}_{\cdots}[l]
&
n \ar@(ul,ur)^{\vep_n} \ar@<-1ex>[l]_{\alpha_{1n}} \ar@<1ex>^{\alpha_{h_nn}}_{\cdots}[l]
}.
\]
Obviously $\rho \in \cI'$. Since $B$ is geometrically irreducible, we know from \cite{BobinskiSchroer2017a}*{Proposition~5.4} that $\cI'$ is generated by relations at most $1$. If $\rho' \in \cI'$, then there exist $\omega_1, \omega_2 \in K Q$ such that $\rho' + \omega_1 e \omega_2 \in \cI$. If additionally we assume $\deg \rho' \leq 1$, then $\omega_1 e \omega_2 = 0$, hence $\rho' \in \cI$. Indeed, if this is not the case, then $\deg (\omega_1 e \omega_2) \geq 2$, which is a contradiction, since there are neither shortcuts nor oriented cycles of positive degree in $Q$. The above implies $\cI' \subseteq (R_0 \cup R_1) = \cJ$. In particular, $\rho \in \cJ$, which finishes the proof.
\end{proof}

\subsection{No overlapping relations}
Roughly speaking, the following lemma says that relations of degree $1$ cannot overlap.

\begin{Lem} \label{lem:nooverlap}
Let $A = K Q / \cI$. Assume $Q$ has three vertices and there are no shortcuts in $Q$. If $A$ is geometrically irreducible, then there exist relations $\rho_1$, \ldots, $\rho_n$ of degree $0$ and a relation $\rho$ of degree $1$ such that
\[
\cI = (\rho_1, \ldots, \rho_n, \rho).
\]
\end{Lem}

\begin{proof}
Since $A$ is geometrically irreducible and $Q$ contains no shortcuts, Proposition~\ref{prop known} implies that up to duality either $Q$ is a subquiver of
\[
Q' : \xymatrix{
0 \ar@(ul,ur)^{\vep_0} & 1 \ar@(ul,ur)^{\vep_1} \ar@<-1ex>[l]_{\alpha_1} \ar@<1ex>^{\alpha_h}_{\cdots}[l] & 2 \ar@<1ex>[l]^{\beta_l} \ar@<-1ex>_{\beta_1}^{\cdots}[l] \ar@(ul,ur)^{\vep_2}}
\]
or a subquiver of
\[
Q'' : \xymatrix{
0 \ar@(ul,ur)^{\vep_0} & 1 \ar@(ul,ur)^{\vep_1} \ar@<-1ex>[l]_{\alpha_1} \ar@<1ex>^{\alpha_h}_{\cdots}[l]  \ar@<1ex>[r]^{\beta_1} \ar@<-1ex>_{\beta_l}^{\cdots}[r] & 2 \ar@(ul,ur)^{\vep_2}}.
\]
We concentrate on the former case.

First observe, that \cite{BobinskiSchroer2017a}*{Proposition~5.4} (or Lemma~\ref{lem:degree}) implies that $\cI$ is generated by relations of degree at most $1$. Moreover, it follows from Theorem~\ref{thm:intro3} that we may assume there is at most one relation $\rho$ such that $s (\rho) = 1$ and $t (\rho) = 0$, and if there is such a relation, then $\vep_0$ and $\vep_1$ belong to $Q$. We have a similar claim for the vertices $1$ and $2$. Consequently, using Theorem~\ref{thm:intro3} again we may assume $Q = Q'$ and $\cI$ is generated by relations $\vep_0^m$, $\vep_1^m$, $\vep_2^m$, $\sum_{i = 0}^{n_1} \vep_0^i \alpha_1 \vep_1^{n_1 - i}$ and $\sum_{j = 0}^{n_2} (\vep_1')^j \beta_1 \vep_2^{n_2 - i}$, for some $0 < n_1, n_2 < m$, where $\vep_1' = \lambda_1 \vep_1 + \cdots + \lambda_{m - 1} \vep_1^{m - 1}$, for some $\lambda_1, \ldots, \lambda_{m - 1} \in K$, with $\lambda_1 \neq 0$. Finally, by symmetry we may assume $n_1 \leq n_2$.

We explain the above claim more precisely. First, we apply Theorem~\ref{thm:intro3} to the algebra $(e_1 + e_2) A (e_1 + e_2)$. In this way we get relations $\vep_1^m$, $\vep_2^m$ and $\sum_{j = 0}^{n_2} \vep_1^j \beta_1 \vep_2^{n_2 - i}$. Next, we apply Theorem~\ref{thm:intro3} to the algebra $(e_0 + e_1) A (e_0 + e_1)$ and get relations $\vep_0^m$, $\vep_1^m$ (note that $m$ does not change) and $\sum_{j = 0}^{n_1} \vep_1^j \beta_1 \vep_2^{n_2 - i}$. However, in this step we may modify $\vep_1$ and that is why we need to introduce $\vep_1'$.

Now consider the dimension vector $\bd = (1, n_2 + 1, 1)$. Then the variety $\rep (A, \bd)$ is described by the conditions $\alpha_1 \vep_1^{n_1}$ and $\vep_1^{n_2} \beta_1$. Let $\cU$ be the subset of $\rep (A, \bd)$ consisting of $M$ such that $M (\vep_1)$ is of rank (at least) $n_2$ (equivalently, $M (\vep_1)$ is similar to $J_{n_2 + 1}$). Then $\cU$ is an (irreducible) open subset of $\rep (A, \bd)$. Moreover, by arguments analogous to those presented in
Subsection~\ref{sub strat}, one calculates that
\[
\dim \cU = n_2^2 + n_2 + h n_2 + k n_2 - n_1 - 2.
\]
On the other hand, if $\cV$ be the subset of $\rep (A, \bd)$ consisting of $M$ such that $M (\vep_1)$ is similar to $J_{(n_2, 1)}$, then again
\[
\dim \cV = n_2^2 + n_2 + h n_2 + k n_2 - n_1 - 2.
\]
Consequently, $\cV$ cannot be contained in the closure of $\cU$, which contradicts irreducibility and finishes the proof in this case.

The case, when $Q$ is a subquiver of $Q''$ is treated in the same way. The only difference is that we do not need to refer to~\cite{BobinskiSchroer2017a}*{Proposition~5.4} in this case.
\end{proof}

\subsection{Proof of Theorem~\ref{thm:conj}}
We present now the proof of Theorem~\ref{thm:conj}. Let $A = K Q / \cI$ be a minimal geometrically irreducible bound quiver algebra and assume there are no shortcuts in $Q$. By definition minimality of $A$ implies that $Q$ is connected. In particular, if there are no arrows in $Q$, then $Q$ has exactly one vertex.

We know from Lemma~\ref{lem:degree} that $\cI$ is generated by relations of degree at most $1$. Fix sets $R_0$ and $R_1$ of relations of degree $0$ and $1$, respectively, such that $R_0 \cup R_1$ generates $\cI$. We also assume that the pair $(R_0, R_1)$ is minimal in the sense, that $|R_1|$ is minimal possible.

Assume first there exists an arrow $\alpha$ in $Q$, which is not a loop, such that there is no $\rho \in R_1$ with $s (\rho) = s (\alpha)$ and $t (\rho) = t (\alpha)$. Then minimality of $A$ implies that $A$ is the path algebra of the quiver
\[
\xymatrix{
i & j \ar[l]_{\alpha}
},
\]
where $i := t (\alpha)$ and $j := s (\alpha)$, hence in particular $Q$ has two vertices. Indeed, if this is not the case, then we put $Q' := (\{ i, j \}, \{ \alpha \})$, $\cI' := 0$, $Q'' := (Q_0, Q_1 \setminus \{ \alpha \})$, and $\cI'' := \cI \cap K Q''$. Then obviously, $Q_0 = Q_0' \cup Q_1''$, $Q_1 = Q_1' \cup Q_1''$, $Q_1' \cap Q_1'' = \emptyset$, and $Q_1' \neq \emptyset \neq Q_1''$ (the last claim uses connectivity again). Moreover, our assumption on $\alpha$ implies $R_0 \cup R_1 \subseteq \cI''$, hence $\cI''$ generates $\cI$ in $K Q$. Altogether the above contradicts minimality of $A$.

Thus we may assume $Q_1 \neq \emptyset$ and for each arrow $\alpha$ in $Q$ there exists $\rho \in R_1$ with $s (\rho) = s (\alpha)$ and $t (\rho) = t (\alpha)$. Fix $\alpha \in Q_1$ and $\rho_1 \in R_1$ such that $s (\rho_1) = s (\alpha)$ and $t (\rho_1) = t (\alpha)$. If $|Q_0| > 2$, then connectivity of $Q$ implies that there exists $\beta \in Q_1$ such that $\{ s (\beta), t (\beta) \} \cap \{ s (\alpha), t (\alpha) \}$ has exactly one element. Fix $\rho_2 \in R_1$ such that $s (\rho_2) = s (\beta)$ and $t (\rho_2) = t (\beta)$. Put $I := \{ s (\alpha), t (\alpha) \} \cup \{ s (\beta), t (\beta) \}$ (note that $|I| = 3$) and $B := A / A e A$, where $e := \sum_{i \not \in I} e_i$. Then $B = K Q' / \cI'$, where $Q'$ is the full subquiver of $Q$ with the vertex set $I$. Note that $\rho_1, \rho_2 \in \cI'$, as there are no shortcuts in $Q$. Since $B$ is geometrically irreducible, Lemma~\ref{lem:nooverlap} implies there exists relations $\rho_1'$, \ldots, $\rho_n'$ of degree $0$ and a relation $\rho$ of degree $1$ such that $\cI' = (\rho_1', \ldots, \rho_n', \rho)$. Since there are neither shortcuts nor oriented cycles of positive degree in $Q$, $\rho_1', \ldots, \rho_n', \rho \in \cI$ (compare the proof of Lemma~\ref{lem:degree}). Consequently, if we put
\[
R_0' := R_0 \cup \{ \rho_1', \ldots, \rho_n' \} \qquad \text{and} \qquad R_1' := (R_1 \setminus \{ \rho_1, \rho_2 \}) \cup \{ \rho \},
\]
then $\cI$ is generated by $R_0' \cup R_1'$. Since $|R_1'| < |R_1|$, this contradicts minimality of the pair $(R_0, R_1)$, hence $|Q_0| \leq 2$.

The list of possible forms of minimal geometrically irreducible algebras without shortcuts is now an immediate consequence of Theorem~\ref{thm:intro3}. \qed

\bibsection

\end{document}